%% file: main.tex
\documentclass[12pt]{amsart}
\pdfoutput=1
\usepackage{graphicx}
\usepackage{amsfonts}
\usepackage{amssymb}
\usepackage{latexsym}
\usepackage{amscd}

\newtheorem{thm}{Theorem}[section]
\newtheorem{cor}[thm]{Corollary}

\newtheorem{prop}[thm]{Proposition}

\newtheorem*{riemann-mapping-theorem}{Riemann Mapping  Theorem}

\newtheorem*{duality}{Duality Theorem}

\newenvironment{pf*}[1]{\proof[#1]}{\endproof}
\usepackage{euscript}

\newcommand{\grad}{\nabla}
\newcommand{\wl}{\widetilde}

\newcommand{\cal}[1]{{\mathcal #1}}

\newcommand{\beq}{\begin{equation}}
\newcommand{\eeq}{\end{equation}}

\theoremstyle{definition}
\newtheorem{defn}{Definition}[section]

\theoremstyle{remark}
\newtheorem{rem}{Remark}[section]

\newcommand{\dist}{\operatorname{dist}}

\renewcommand{\mod}{\operatorname{mod}}
\newcommand{\tl}{\tilde}

\newcommand{\eps}{\epsilon}

\numberwithin{equation}{section}
\newcommand{\thmref}[1]{Theorem~\ref{#1}}
\newcommand{\propref}[1]{Proposition~\ref{#1}}

\newcommand{\figref}[1]{Figure~\ref{#1}}

\newcommand{\cO}{{\cal O}}
\newcommand{\cI}{{\cal I}}
\newcommand{\cA}{{\cal A}}

\newcommand{\cT}{{\cal T}}

\newcommand{\cP}{{\cal P}}
\newcommand{\cC}{{\cal C}}

\newcommand{\cR}{{\cal R}}

\newcommand{\cD}{{\cal D}}
\newcommand{\cE}{{\cal E}}
\newcommand{\cS}{{\cal S}}

\newcommand{\RR}{{\mathbb R}}

\newcommand{\TT}{{\mathbb T}}
\newcommand{\ZZ}{{\mathbb Z}}
\newcommand{\NN}{{\mathbb N}}

\newcommand{\QQ}{{\mathbb Q}}

\newcommand{\ignore}[1]{{}}

\newcommand{\cir}[1]{\overset{\circ}{#1}}

\begin{document}

\title[Hyperbolicity of renormalization]{Hyperbolicity of renormalization of circle maps with a break-type singularity}

\date{\today}

\author{K. Khanin, M. Yampolsky}
\thanks{The authors were partially supported by NSERC Discovery Grants}
\begin{abstract}We study the renormalization operator of circle homeomorphisms with a break point and show that it possesses a hyperbolic horseshoe attractor.
\end{abstract}

\maketitle
\input{intro}
\input{expansion}

\input{bib}
\end{document}

%% file: intro.tex
\section{Preliminaries}

\subsection*{Renormalization of homeomorphisms of the circle with a break singularity.}
This work concerns renormalization of  homeomorphisms of the circle $\TT=\RR/\ZZ$ with a break singularity. Specifically, 
these are mappings
$$f:\TT\to\TT$$
 with the following properties:
\begin{itemize}
\item $f\in C^{2+\eps}(\TT\setminus \{0\})$ for some $\eps>0$, and $f\in C^0(\TT)$;
\item $f'(x)>0$ for all $x\neq 0$;
\item $f$ has  one-sided derivatives $f'(0+)>0$ and $f'(0-)>0$ and 
$$f'(0-)/f'(0+)=c^2\neq 1.$$
\end{itemize}

Renormalizations of such maps were extensively studied by the first author and others (see e.g. \cite{KV, KK, KT2}).
We recall the definition of renormalization of a circle homeomorphism $f$ at a point $x_0\in\TT$ (to fix the ideas we will 
set  $x_0=0$)
very briefly, the reader will find a detailed account in any of the above references.
Firstly, let us denote $\rho(f)$ the rotation number of $f$. Denote $\bar f:\RR\to\RR$ the lift of $f$ to the real line via the 
projection $x\mapsto x\mod\ZZ$ with the 
property $\bar f(0)\in[0,1)$. Assume that $\rho(f)\neq 0$, and consider the smallest $r_0\in\NN$ for which 
$$1\in [\bar f^{r_0}(0),\bar f^{r_0+1}(0)).$$
Denote $I_0=[0,\bar f(0)]$, $I_1=[\bar f^{r_0}(0)-1,0]$ and let $\zeta_1:I_0\cup I_1\to I_0\cup I_1$ be the pair of interval homeomorphisms 
$$\zeta_1=(({\bar f}^{r_0}-1)|_{I_0},\bar f|_{I_1}).$$
The first renormalization $\cR(f)=(\eta_1,\xi_1)$ is the rescaled pair $\alpha_1\circ \zeta_1\circ \alpha_1^{-1}$ where $\alpha_1$ is the orientation-reversing  
rescaling $x\mapsto -x/\bar f(0).$ Thus, 
$$\eta_1:[-1,0]\to [-b_1,a_1]\text{ and }\xi_1:[0,a_1]\to [-1,-b_1],$$
where $-b_1=\alpha_1(\bar f^{r_0+1}(0)-1)\in [-1,0]$ and $a_1=\alpha_1(\bar f^{r_0}(0)-1)>0.$

We can now proceed to inductively define a  finite or infinite sequence of renormalizations $\cR^n(f)=(\eta_n,\xi_n)$ as follows. 
Consider the smallest $r_n\in\NN$ such that 
$$0\in [\eta_n^{r_n}(\xi(0)),\eta^{r_n+1}(\xi(0))).$$
We refer to $r_n$ as the {\it height} of $\zeta_n=(\eta_n,\xi_n)$.
If such a number does not exist, then $f$ is renormalizable only $n$ times; set $r_n=\infty$ and terminate the sequence.
Otherwise, set $\alpha_{n+1}(x)=-x/\xi_n(0)$ and let
$\cR^n f=(\eta_{n+1},\xi_{n+1}),$ where
$$\eta_{n+1}=\alpha_{n+1}\circ \eta_n^{r_n}\circ \xi_n\circ \alpha_{n+1}^{-1}:[-1,0]\to [-b_{n+1},a_{n+1}]\text{ and }$$
$$\xi_{n+1}=\alpha_{n+1}\circ \eta_{n}\circ \alpha_{n+1}^{-1}:[0,a_{n+1}]\to [-1,b_{n+1}].$$
It will also be convenient for us to introduce the $n$-th {\it pre-renormalization}
\begin{equation}
\label{prerenorm}
p\cR^n f\equiv (\gamma_n^{-1}\circ \eta_n\circ \gamma_n,\gamma_n^{-1}\circ \xi_n\circ \gamma_n),
\end{equation}
where $$\gamma_n\equiv \alpha_n\circ \alpha_{n-1}\circ\cdots\circ \alpha_1.$$ Thus, 
$p\cR^n f$ is a composition of iterates of the original pair $\zeta_1$, and $\cR^n f$ is its suitable rescaling.

\begin{figure}[htb]
\centerline{\includegraphics[width=1.2\textwidth]{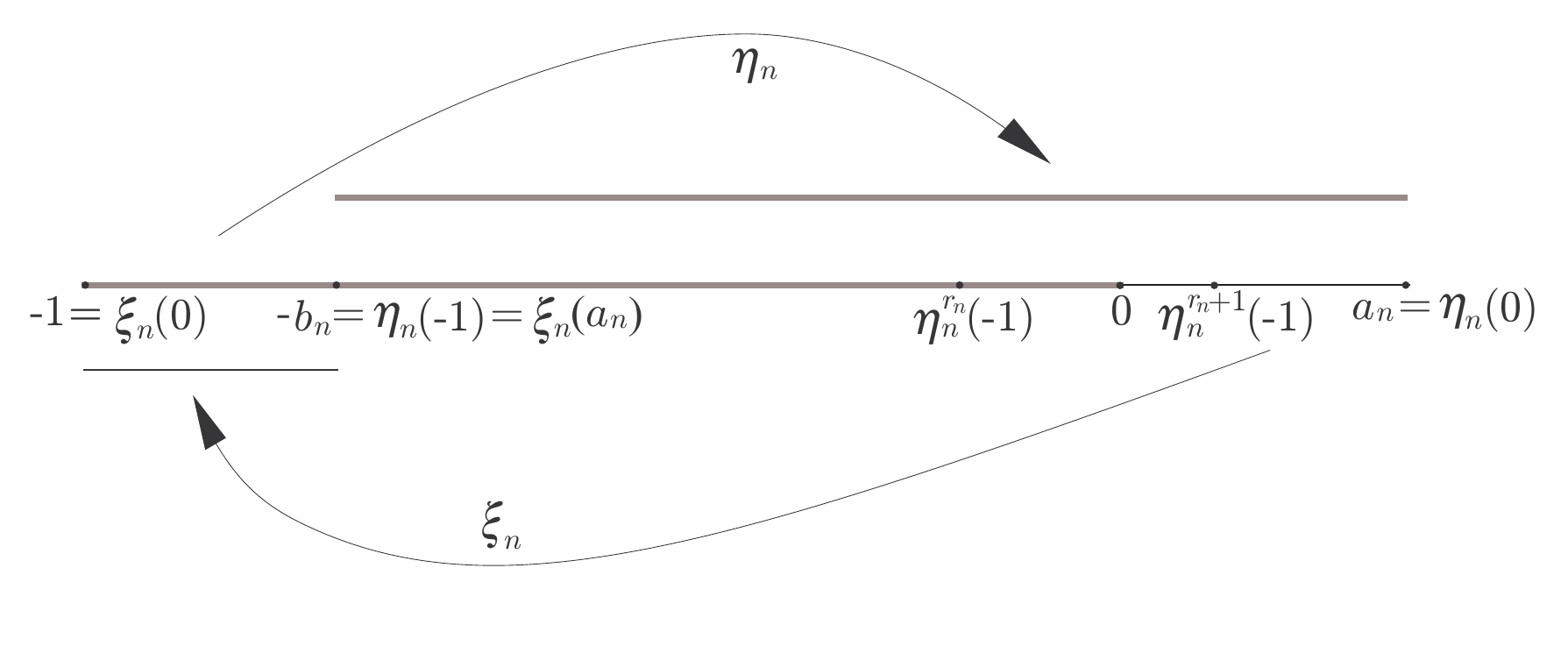}}
\caption{\label{pair}The dynamics of the pair $(\eta_n,\xi_n)$.
}
\end{figure}

Using the convention $1/\infty=0$, we recover $\rho(f)$ as the finite or infinite continued fraction
$$\rho(f)=[r_0,r_1,r_2,\ldots].$$
Note that one advantage of defining the continued fraction via the dynamics of $f$ as above is that we obtain a canonical expansion for all rational rotation numbers
 $\rho(f)$.

\subsection*{The invariant family of M{\"o}bius transformations.}
Fix a value of $c\neq 1$, and define two families of  M{\"o}bius transofrmations
$$F_{a,v,c}(z)=\frac{a+cz}{1-vz},\text{ and }G_{a,v,c}(z)=\frac{a(z-c)}{ac+z(1+v-c)}.$$
Set
$$\zeta_{a,v,c}\equiv(F_{a,v,c},G_{a,v,c}):[-1,a]\to [-1,a].$$
The parameter $c$ can be read off from
\begin{equation}
\label{eq1}
c^2=\frac{F_{a,v,c}'(0)G'_{a,v,c}(F_{a,v,c}(0))}{G_{a,v,c}'(0)F'_{a,v,c}(-1)}.
\end{equation}
Observe that 
\begin{equation}
\label{eq2}
-b\equiv F_{a,v,c}(-1)=G_{a,v,c}(a)=\frac{a-c}{1+v},\; F(0)=a,\;G(0)=-1.
\end{equation}

\noindent
We set the following additional constraints:
\begin{itemize}
\item[(I)] $c\geq a>0$;
\item[(II)] $b=(c-a)/(v+1)\in [0,1)$.
\end{itemize}

\noindent
Together with equalities (\ref{eq2}), the above conditions imply (see Lemma 4. in \cite{KK}):
\begin{prop}
\label{homeo}
The maps 
$$F_{a,v,c}:[-1,0]\to [-b, a]\text{ and }G_{a,v,c}:[0,a]\to [-1,-b]$$
are orientation preserving homeomorphisms. Further,
$$F'_{a,v,c}(z)>0\text{ for }z\in[-1,0],\text{ and }G'_{a,v,c}(z)>0\text{ when }z\in[0,a].$$
\end{prop}

Denote
$$\pi:\RR\to \RR/\ZZ\text{ the projection }x\mapsto x \mod \ZZ,$$
and set 
$$z\equiv -\frac{a}{c}= (F_{a,v,c})^{-1}(0)\in[-1,0].$$
Identifying the circle $\RR/\ZZ$ with $[-1,0]$ via the projection $\pi$, we can
view the pair of maps
$$\tau_{a,v,c}\equiv (F_{a,v,c}|_{[-1,z]},G_{a,v,c}\circ F_{a,v,c}|_{[z,0]}):[-1,0]\mapsto [-1,0]$$
as a circle homeomorphism
\begin{equation}\label{circle map}
f_{a,v,c}(z)=\pi\circ\tau_{a,v,c}\circ \pi^{-1}.
\end{equation}
We will denote its rotation number 
$$\rho(f_{a,v,c})\equiv \rho(\zeta_{a,v,c})=[r_0,r_1,\ldots].$$
In the case when $r_0\neq \infty$, we define the pre-renormalization 
$$p\cR\zeta_{a,v,c}\equiv (F_{a,v,c}^{r_0}\circ G_{a,v,c}|_{[0,a]},F_{a,v,c}|_{[F_{a,v,c}^{r_0}\circ G_{a,v,c}(0),0]}).$$

The following invariance property of the family $(F_{a,v,c},G_{a,v,c})$ is of a key importance:
\begin{prop}
\label{invariance}
Suppose $\rho (\zeta_{a,v,c})\neq 0$. Then the
 renormalization of a pair $\zeta_{a,v,c}=(F_{a,v,c}(z),G_{a,v,c}(z))$ is of the form $\zeta_{a',v',1/c}$ for some $a'$, $v'$.
\end{prop}

Further, the following holds (see \cite{KV,KV2}):
\begin{thm}
\label{convergence}
Let $f$ be a circle mapping with a break singularity of class $C^{2+\eps}$, with break size $c$. Let $\zeta_n=(\eta_n,\xi_n)=\cR^n(f)$,
and $a_n=\eta_n(0)$, $-b_n=\xi_n(0)$ as above. Set $c_n=c$ for $n$ even and $c_n=1/c$ for $n$ odd, and
$v_n=(c_n-a_n-b_n)/b_n$.
Then there exist constants $C>0$ and $0<\lambda<1$ such that
$$||\eta_n-F_{a_n,v_n,c_n}||_{C^2([-1,0])}\leq C\lambda^n\text{ and }
||\xi_n-G_{a_n,v_n,c_n}||_{C^2([0,a_n])}\leq C\lambda^n/a_n\text.$$
\end{thm}

At this point it is instructive to draw parallels with the theory of renormalization of critical circle maps 
(see \cite{Ya1,Ya2,Ya3} and references therein).
The   $\cR$-invariant two-dimensional parameter space $(a,v)$ of pairs of M{\"o}bius maps is an analogue of the Epstein class $\cE$ of 
critical circle maps. Similarly to \propref{convergence}, renormalizations of smooth critical circle maps converge to $\cE$ geometrically fast. The maps in $\cE$ 
are analytic and possess rigid global structure, yet the Epstein class does not embed into a  finite-dimensional space. In fact, as shown in \cite{Ya2} by the second
author, suitably defined conformal conjugacy classes of maps in $\cE$ form a Banach manifold. 
This has required developing an extensive analytic machinery to tackle renormalization convergence and hyperbolicity in $\cE$ (see \cite{Ya3}, which is the
culmination of this work, and references therein). In contrast, the finite dimensionality of the space $(a,v)$ 
allows for a direct approach to proving hyperbolicity of $\cR$ restricted to this invariant space.

Note that since maps with a break singularity are only $C^0$ at the origin, it does not make sense to speak of renormalization of {\it commuting} pairs,
the way one does for critical circle maps. We note, however, that maps $F_{a,v,c}$ and $G_{a,v,c}$ satisfy the following commutation condition:
\begin {equation}
\label{commute}
G(F(z))=F(G(c^2 z)).
\end{equation}
While one can readily check this algebraically, \thmref{convergence} leads to the same conclusion. Indeed, compare the value of the iterate
$f^{q_n+q_{n+1}}$ to the left and to the right of zero. Since
$f'(0-)/f'(0+)=c^2$, we have (in a self-explanatory notation),
$$f^{q_n+{q_{n+1}}}(0-)=f^{q_n+{q_{n+1}}}(c^2\cdot 0+).$$
Since the compositions $\eta_n\circ \xi_n$ and $\xi_n\circ \eta_n$ are both obtained by linearly rescaling the above iterate (on different sides of $0$), the equality (\ref{commute}) holds for limits of renormalizations.

\subsection*{Renormalization on the space of M{\"o}bius pairs.}
Let us denote
 $$\cD_c=\{(a,v)\;|\;0<a\leq c\text{ and }a+v>c-1\}.$$
The set $\cO^n_c\subset \cD_c$ consists of pairs $(a,v)$ for which the pair $\zeta_{a,v,c}=(F_{a,v,c},G_{a,v,c})$ is $n$-times renormalizable.
The infinite intersection
$$\cO^\infty_c=\cap_{n\in\NN} \cO^n_c$$
is the set of infinitely renormalizable pairs. It can be equivalently characterized as the set of pairs $\zeta_{a,v,c}$ with
an irrational rotation number. 
The set of renormalizable pairs $\cO^1_c$  is naturally stratified into a collection of subsets 
$$\Pi_{k,c}=\{(a,v)\text{ such that }\zeta_{a,v,c}\text{ has height }r=k\in\NN\}.$$
In particular, for $(a,v)\in\Pi_{k,c}$ the rotation number $\rho(\zeta_{a,v,c})$ can be expanded into a continued fraction
of the form $[k,\ldots]$.

We  define the renormalization operator 
$$\cR_c:\cO^1_c\to \cD_{1/c},$$
as $(a,v)\to(a',v')$ where $\zeta_{a',v',1/c}=\cR_c(\zeta_{a,v,c})$.
The operator $\cR_c$ 
is analytic on the interior of each of the sets $\Pi_{k,c}$. It is also convenient for us to define an operator 
$$\cT_c\equiv \cR_{1/c}\circ \cR_c,$$
so that 
$$\cT_c:\cO^2_c\to \cD_c.$$
We will use notations $A_{k,c}(a,v)$ and $V_{k,c}(a,v)$ for
\begin{equation}
\label{eq-av}
(A_{k,c}(a,v),V_{k,c}(a,v))=\cT^k_c(a,v).
\end{equation}
As in (\ref{prerenorm}), we will define the pre-renormalizations $p\cR_c(a,v)$ and $p\cT_c(a,v)$ as the non-rescaled
iterates of the pair $\zeta_{a,v,c}$ corresponding to the appropriate renormalization.

As shown in \cite{KK} (Lemma 5),
\begin{prop}
\label{nonrenorm}
If $c<1$ then the set of non-renormalizable parameters $\cD_c\setminus \cO^1_c$ is empty. If $c>1$ then
$$\cD_c\setminus \cO^1_c=\left\{(a,v)|\; \max\{ 0,c-v-1\}<a\leq \frac{(c-1)^2}{4v},\;v>\frac{c-1}{2}\right\}.$$
\end{prop}

\noindent

\noindent
As an example, in \figref{range} we picture some of the above described sets for $c=3$. The set 
$$\cD_3=\{(a,v)\subset \RR^2|\;0<a\leq 3,\;v>2-a\}$$
is pictured together with the first few $\Pi_{k,3}$.
The  complement of the region $\cO^1_3$ is also indicated. The curved portion of  the boundary of $\cO^1_3$ consists of parameters for which $F$ has a fixed point
$w\in(-1,0)$ with $F'(w)=1$; its equation is $v=1/a$ for $a\in (0,1)$.

\begin{figure}[h]
\centerline{\includegraphics[width=0.9\textwidth]{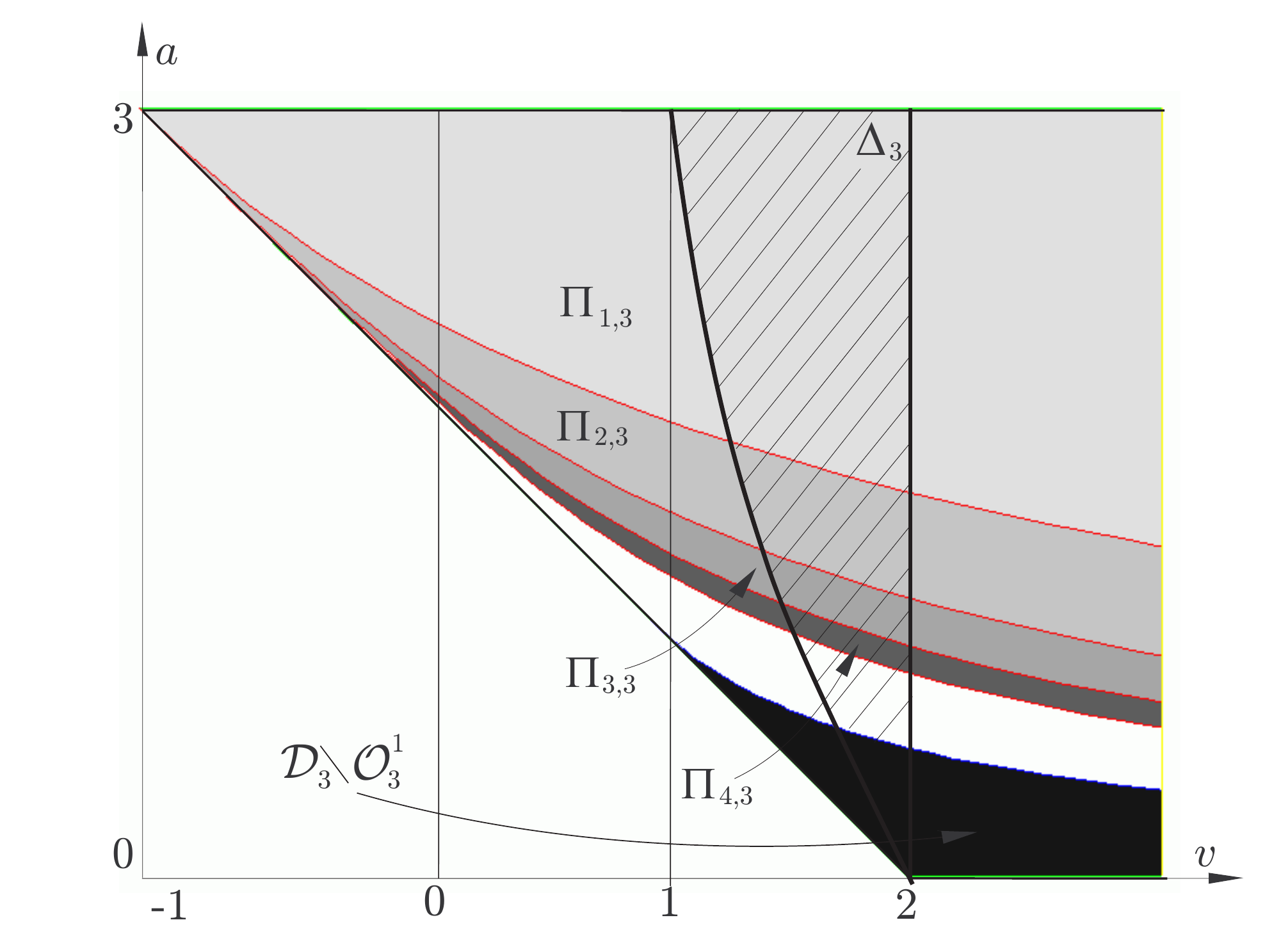}}
\caption{\label{range}The set $\cD_3$ and $\Pi_{k,3}$ for $k=1,2,3,4$.
}
\end{figure}

As shown in \cite{KK}, orbits of renormalization eventually fall into a compact subset of $\cD_c$. 
Namely, for $c>1$ we define
$$\Delta_c=\{ (a,v)\in\cD_c|\;0\leq v\leq c-1\},$$
and for $c<1$ we set
$$\Delta_c=\{ (a,v)\in\cD_c|\;c-1\leq v\leq 0\}.$$
Then we have:
\begin{prop}
\label{invariant domain}
The image
$$\cT_c(\Delta_c\cap \cO^2_c)\subset \Delta_c.$$
Further, for every $(a,v)\in \cO^\infty_c$ there exists $k\geq 0$ such that
$$\cT_c^k(a,v)\in \Delta_c.$$
\end{prop}

The following discovery of \cite{KT2} is going to be key for our study. Denote
$$\cI_c(a,v)=\left( \frac{c-1-v}{av},-\frac{v}{c}\right),\; \text{Jac}(I_c)=\frac{c-1-v}{a^2cv}>0\text{ for }(a,v)\in\cD_c.$$
It is easy to verify that
$$\cI_c:\cD_c\to \cD_{1/c}.$$
This map is an involution in the following sense:
$$\cI_{1/c}\circ \cI_c=\text{Id}.$$
Then the following is true:
\begin{duality} \cite{KT2}
The renormalization operator 
$$\cR_c:\cO^1_c\to\cD_{1/c}$$
is injective, and 
$$\cR_c^{-1}=\cI_{1/c}\circ \cR_c\circ \cI_{1/c},$$ where the left-hand side is defined.
Furthermore, set $(a',v')=\cI_c(a,v)$. Then the pairs
$\zeta_{a,v,c}$ and 
$\cR_{c}^{-1}\zeta_{a',v',1/c}$ have the same height.
\end{duality}

As a consequence, we have 
\begin{equation}
\label{t conj}
\cT_c^{-1}=\cI_{1/c}\circ \cT_{1/c}\circ \cI_c=(\cI_c)^{-1}\circ \cT_{1/c}\circ \cI_c.
\end{equation}

\section{Previous results and the statement of the main theorem}

Denote $\Sigma_\NN$ the set of bi-infinite sequences $(r_i)$, $r_i\in\NN$, $i\in\ZZ$, endowed with the distance
$$d((r_i),(t_i))=\sum_{n\in\ZZ}\left|\frac{1}{r_i}-\frac{1}{t_i}\right|\cdot 2^{-|n|} .$$
Set $$\sigma(r_i)_{i\in\ZZ}=(r_{i+1})_{i\in\ZZ}.$$
The following result was shown by the first author and Teplinski in \cite{KT2}:

\begin{thm}
\label{thm-1}
Let $c>1$. There exists a set $\cA_c\subset \Delta_c$ and $\lambda>1$ such that the following properties hold:
\begin{itemize}
\item for any pair of values $(a_i,v_i)\in \cO^\infty_c$,  $i=1,2$ with $\rho(\zeta_{a_1,v_1,c})=\rho(\zeta_{a_2,v_2,c})$
there exists $K_1>0$ such that
$$||\cT_c^n(a_1,v_1)-\cT_c^n(a_2,v_2)||<K_1\lambda^{-n} ;$$
\item for any $(a,v)\in\cO^\infty_c$ there exists $K_2>0$ such that
$$\dist_{\RR^2}(\cT_c^n(a,v),\cA_c)<K_2\lambda^{-n};$$
\item finally, there exists a homeomorphism $\iota:\Sigma_\NN\to \cA_c$ such that
$$\rho(\zeta_{\iota(r_i),c})=[r_0,r_1,\ldots]\text{ and }\iota\circ \cT_c\circ \iota^{-1}=\sigma^2.$$
\end{itemize}
\end{thm}

\noindent
In the present work we will use a different approach to prove the above result for $c\in(0.5,2)$. 
We will further show that the attractor
$\cA_c$ is hyperbolic:

\begin{thm}
\label{thm-2}
Let $c\in(0.5,2)\setminus \{ 1\}$. There exists a $D\cT_c$-invariant splitting $E^u_z\oplus E^s_z$ of the tangent bundle $T_z\Delta_c$ over $\cA_c$ and $k_0\in\NN$ such that
$$||D\cT^{k_0}_cv||>\lambda||v||\text{ for }v\in E^u_z\text{ and }||D\cT^{k_0}_cv||<\lambda^{-1}||v||\text{ for }v\in E^s_z.$$

\end{thm}

The reasons for the restriction $c\in(0.5,2)$ are purely technical, as described in Remark \ref{rem:explanation}.
We conjecture that the statement of \thmref{thm-2} holds for all positive values of $c\neq 1$. 

Stable manifolds of points in $\cA_c$ are sets $L_{c,\rho}$ consisting of the pairs $(a,v)$ with 
$$\rho(\zeta_{a,v,c})=\rho\notin\QQ.$$
One of the corollaries of \thmref{thm-2} is that $L_{c,\rho}$ are $C^\omega$-smooth curves (\propref{prop:analytic}).  
This is an improvement on \cite{KT2}, where it was shown that for $\rho\notin\QQ$ the set $L_{c,\rho}$ is a $C^0$-curve.

Let us recall that an irrational number $\rho=[r_0,r_1,r_2,\ldots]$ is of {\it type bounded by $B$} if $\sup r_i\leq B<\infty$.
Let us denote $\cA_c^B$ the compact subset of $\cA_c$ consisting of maps of type bounded by $B$.

\begin{rem}
\label{rem:bdedtype}
For every $B\in\NN$ we obtain Theorems \ref{thm-1} and \ref{thm-2} for $\cA_c^B$  independently of \cite{KT2}. The proofs for the whole of
$\cA_c$ use a result of \cite{KT2} on the geometry of the curves $L_{c,\rho}$.
\end{rem}

%% file: expansion.tex
\section{Expansion of renormalization}

In this section we will describe the expansion properties of renormalization.
We begin by defining a subset $\cC$ in the tangent bundle  $T\Delta_c$ as follows. 
As before, for a pair $(F_{a,v,c},G_{a,v,c}):[-1,a]\to [-1,a]$ let us define 
$$\tau_{a,v,c}\equiv (\wl F_{a,v,c}, \wl G_{a,v,c})$$
 to be the first return map of the interval $[-1,0]$:
$$\wl F_{a,v,c}\equiv F_{a,v,c}|_{[-1,F_{a,v,c}^{-1}(0)]},\; \wl G_{a,v,c}\equiv G_{a,v,c}\circ F_{a,v,c}|_{[F_{a,v,c}^{-1}(0),0]}.$$
We set
$$\cC_{a,v}=\{\bar v=(\alpha,\nu)\; |\;\min(\inf_{x\in[-1,F_{a,v,c}^{-1}(0)]}\nabla_{\bar v}\wl F_{a,v,c},\inf_{x\in[F_{a,v,c}^{-1}(0),0]}\nabla_{\bar v}\wl G_{a,v,c})>0\}\subset T_{a,v}\Delta_c.$$

It is obvious from the definition, that:

\begin{prop}
\label{properties cone}
For every $(a,v)\in \Delta_c$ the set $\cC_{a,v}$ is an open cone in $T_{a,v}\Delta_c$.
\end{prop}

We next prove:
\begin{prop}
\label{rho-monotone}
Let $\bar \gamma(t)=(a(t),v(t)):(0,1)\to\Delta_{c}$ be a smooth curve with the property 
$$\frac{d}{dt}\bar \gamma(t)\in \cC_{\bar\gamma(t)}\text{ for all }t.$$
Then the function
$$\rho(t)\equiv \rho(\zeta_{a(t),v(t),c})$$
is non-decreasing. Furthermore, if $\rho(t_0)\notin \QQ$ then $\rho(t)$ is strictly increasing at $t_0$.
\end{prop}
\begin{proof}
For $(a(t),v(t))\in\bar\gamma(t)$ set $\zeta_t\equiv\zeta_{a(t),v(t),c}$. Fix $t_0\in(0,1)$ and let $\zeta_{t_0}^k(0)\neq 0$ 
be a closest return of $0$
under the dynamics of the pair $\zeta_{t_0}$. 
An easy induction shows that $\frac{d}{dt}\zeta_t^k(0)|_{t=t_0}$ is positive. Thus, the 
heights $r_{2i}$ of renormalizations $\cR^{2i}\zeta_t$ {\it decrease}, and the heights $r_{2i+1}$ of renormalizations $\cR^{2i+1}\zeta_t$
{\it increase} with $t$. Hence, the value of the rotation number $\rho=[r_0,r_1,\ldots]$ is a non-decreasing function of $t$.
The last assertion is similarly evident and is left to the reader. 
\end{proof}

\subsection*{The cone $\cC_{a,v}$ is not empty.}
Let us show that for every $c\in(0.5,2)$ and for every $(a,v)\in \overset{\circ}{\Delta_c}$ there is a non-zero tangent vector inside $\cC_{a,v}$.
We identify such directions explicitly in the statements below.

\begin{prop}
\label{a-monotone-1}
Fix $c>1$. For every pair $(a,v)\in\overset{\circ}{\Delta_c}$ 
set $\bar v= \frac{\partial}{\partial a}\in T_{a,v}\Delta_c$. Then
$$\inf_{z\in[-1,0]}\grad_{\bar v} { F_{a,v,c}(z)}>\frac{1}{c}>0,\text{ and }\grad_{\bar v} G_{a,v,c}(z)>0\text{ for all }z\in(0,a].$$
\end{prop}
\begin{proof}
The computations are easy:
$$\frac{\partial F_{a,v,c}(z)}{\partial a}
=\frac{1}{1-vz}.$$
Noting that $c-1> v> 0$, and $z\in[-1,0]$ we get 
$$\inf_{z\in[-1,0]}\frac{\partial F_{a,v,c}(z)}{\partial a}>\frac{1}{c}.$$
Further,
$$\frac{\partial G_{a,v,c}}{\partial a}=\frac{(z-c)z(1+v-c)}{(ac+z(1+v-c))^2}.$$ 
To estimate the numerator, note that $z\leq a<c$, so $z-c<0$ and 
$v<c-1$, so $1+v-c<0$. Hence,
$$\frac{\partial G_{a,v,c}}{\partial a}>0 \text{ for }z\in(0,a].$$

\end{proof}



 Figure \ref{monotone} illustrates the above monotonicity property for $c>1$. 

\begin{figure}[h]
\centerline{\includegraphics[width=0.7\textwidth]{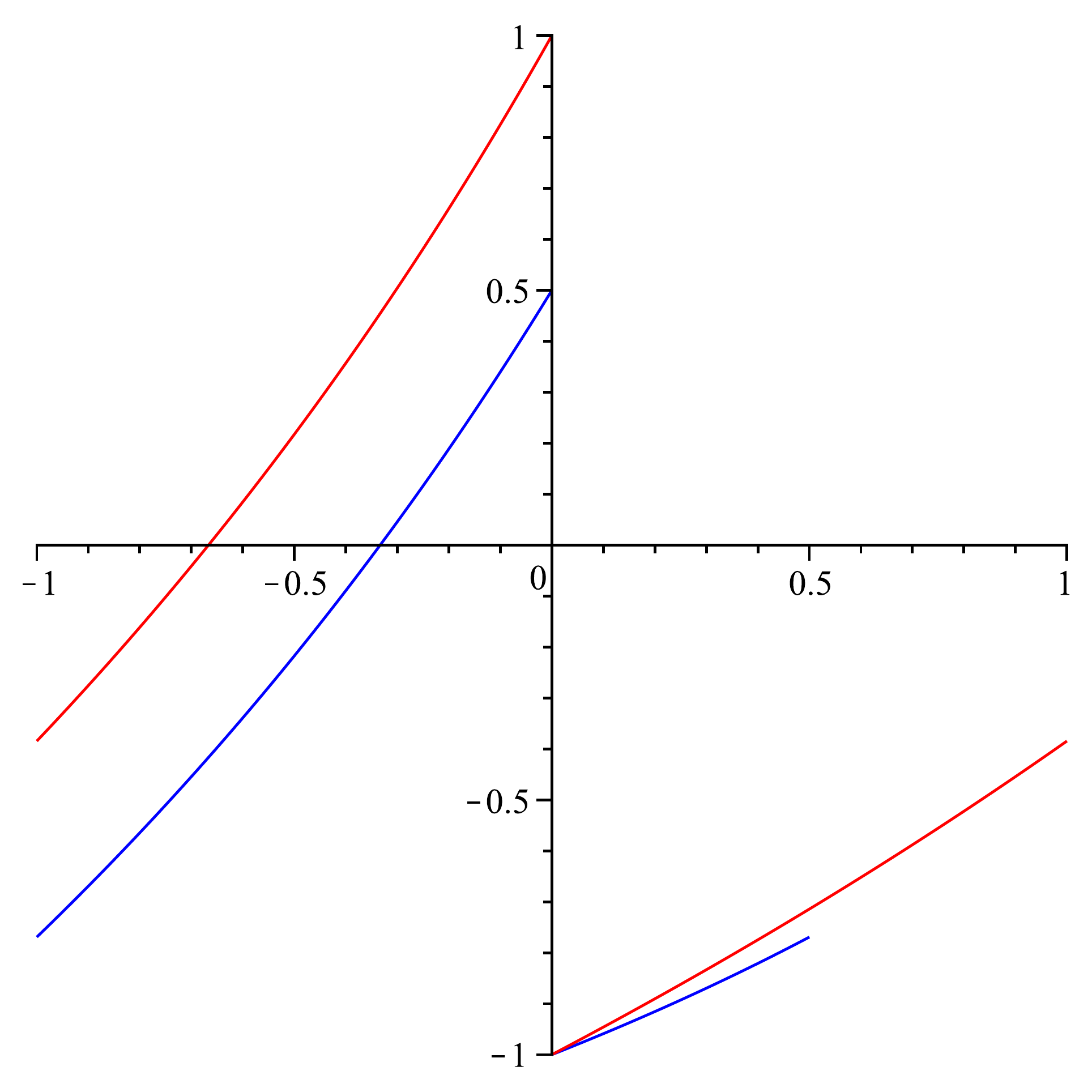}}
\caption{\label{monotone}The graphs of $F_{a,v,c},G_{a,v,c}$ for $c=1.5$, $v=0.3$, and $a=1$ (above) and $a=0.5$ (below).
}
\end{figure}

\begin{prop}
\label{monotone-4}
Suppose $c\in(0.5,1)$. 
For every pair $(a,v)\in\overset{\circ}{\Delta_c}$
set $\bar v= a\frac{\partial}{\partial a}+c\frac{\partial}{\partial v}$. Then
$$\inf_{z\in[-1,0]}\grad_{\bar v} { F_{a,v,c}(z)}>0,\text{ and }\grad_{\bar v} G_{a,v,c}(z)>0\text{ for all }z\in(0,a].$$
\end{prop}
The similarly explicit computation is left to the reader.
As a corollary of the Chain Rule, we have:
\begin{prop}
\label{monotone-4a}
Suppose $(a,v)\in\cir{\Delta}_c$ and $\bar v\in T_{a,v}\Delta_c$ with the properties
$$\inf_{z\in[-1,0]}\grad_{\bar v} { F_{a,v,c}(z)}>0,\text{ and }\grad_{\bar v} G_{a,v,c}(z)>0\text{ for all }z\in(0,a].$$
Then $\bar v\in\cC_{a,v}$.
\end{prop}

We remark:
\begin{rem}
\label{rem:explanation}
We do not know if the cone field $\cC$ is non-empty when the restriction $c>0.5$ is removed. 
This is the only place in the proofs where this restriction is needed.

Moreover, it would be sufficient for our
purposes to find tangent vectors which ``lift'' first return maps of renormalization intervals of a deeper level,
instead of the first return map of $[-1,0]$. However, if $\tl F$ and
$\tl G$ are replaced by a composition corresponding to such a return map, an
explicit calculation of the directional derivative similar to the ones above becomes quite involved, and a brute force approach
to finding a non-zero tangent vector with the desired property may become impractical.
\end{rem}

\subsection*{The expansion properties of the cone field $\cC$.}
We begin by recalling
how the composition operator acts on vector fields. For a pair of smooth functions $f$ and $g$, denote
$$\text{ Comp}(f,g)=f\circ g,$$
viewed as an operator $C^\infty\times C^\infty\to C^\infty$ and let $D\text{Comp}$ denote its differential. An elementary calculation shows that 
\begin{equation}
\label{eqn-comp}
D\text{Comp}|_{(f,g)}:(\phi,\gamma)\to f'\circ g\cdot \gamma+\phi\circ g.
\end{equation}
The significance of the formula (\ref{eqn-comp}) for us lies in the following trivial observation: if $f$ and $g$ are both increasing functions,
and the vector fields $\phi$ and $\gamma$ are non-negative, then
\begin{equation}
\label{eqn-comp2}
\inf_x D\text{Comp}|_{(f,g)}(\phi,\gamma)\geq \inf_x\phi.
\end{equation}
For a pair $\zeta_{a,v,c}$ with $(a,v)\in\cO^2_c$ let us set 
$$\lambda_{a,v,c}=F_{a,v,c}^{r_0}(-1)<0,$$
where, as before, $r_i$ denotes the height of $\cR^i\zeta_{a,v,c}$.
Note that
\begin{equation}
\label{def-r2}
\begin{split}
\cT\zeta_{a,v,c}=\Bigg( & - \frac{1}{\lambda_{a,v,c}}(F_{a,v,c}^{r_0}\circ G_{a,v,c})^{r_1}\circ F_{a,v,c}(-\lambda_{a,v,c} x) ,\\
& -\frac{1}{\lambda_{a,v,c}}F_{a,v,c}^{r_0}\circ G_{a,v,c}(-\lambda_{a,v,c} x)\Bigg).
\end{split}
\end{equation}

We will require the following standard real {\it a priori} bound (see \cite{KT1}):
\begin{prop}
\label{rab}
There exists $\delta>0$ such that for every $(a,v)\in\cO^{2}_c$,
$$\lambda_{a,v,c}>-\frac{1}{1+\delta}.$$
\end{prop}

A key point for us is the following statement:

\begin{prop}
\label{expcone1}
There exist $k\in\NN$ and $\delta>0$ such that the following holds. Let $(a,v)\in\cO_c^{2k}$ and let $\bar v\in\cC_{a,v}$.
Then
$$\nabla_{\bar v} A_{k,c}(a,v)>C(1+\delta)^k,$$
where the constant $C>0$ depends only on $\bar v$.
\end{prop}
\begin{proof}
Let $\bar v=(\alpha,\nu)\in \cC_{a,v}$. Consider a smooth deformation 
\begin{equation}
\label{deform}
\zeta_t^{\bar v}=(F_{a+\alpha t,v+\nu t,c},G_{a+\alpha t,v+\nu t,c})\equiv (F_t,G_t).
\end{equation}
For $m\in\NN$ let us denote  $$\cT^{m}\zeta_t^{\bar v}\equiv (F^t_{m},G^t_m)\text{, and }p\cT^m\zeta_t^{\bar v}\equiv(H^t_m,K^t_m).$$
Let $$\lambda_m^t\equiv |K^t_m(0)|.$$
An easy induction shows that 
\begin{itemize}
\item[(a)] $F^t_{k}(x)=\frac{1}{\lambda_{k}^t} H^t_k\circ (\lambda_{k}^t x);$
\item[(b)] $H_k^t(0)>0$.
\end{itemize}
A repeated application of (\ref{eqn-comp}) 
implies that 
\begin{itemize}
\item[(c)] $\frac{\partial}{\partial t}H^t_{k}(x)>\eps$ where $\eps=\inf_{x\in[0,a]}\nabla_{\bar v}F(x)$;
\item[(d)] $\frac{d}{d t}\lambda_{k}^t<0$.
\end{itemize}
We calculate:
$$\frac{\partial}{\partial t}\left( \frac{1}{\lambda_k^t}H_k^t(\lambda_k^t x) \right)=-\frac{\frac {d}{dt}\lambda_k^t}{(\lambda_k^t)^2}H^t_k(\lambda_k^t x)+\frac{1}{\lambda_k^t}\left(\frac{\partial H_k^t(\lambda_k^t x)}{\partial t}+\frac{\partial H_k^t(x)}{\partial x}\frac{d \lambda_k^t}{d t}x\right).$$
Substituting $x=0$ and using $(a)-(d)$ we see that  
$$\left. \frac{\partial}{\partial t}\right| _{t=0}\left. \left( \frac{1}{\lambda_k^t}H_k^t(\lambda_k^t x) \right)\right| _{x=0}=\nabla_{\bar v}A_{k,c}(a,v)\geq \frac{1}{\lambda_k^0}\eps.$$
Using the real {\it a priori} bound from \propref{rab} completes the proof.
\end{proof}

\ignore{
\begin{prop}
\label{invariant cone}
The cone field $\cC$ is $\cT_c$-invariant:
\begin{itemize}
\item[(a)] $D\cT_c:\cC\to \cC.$
\end{itemize}
Moreover, let $\bar v\neq 0$ belong to the closure of the cone $\cC_{a,v}$. Then 
\begin{itemize}
\item[(b)] $D\cT|_{(a,v)}\bar v\in \cC_{a,v}.$
\end{itemize}
Further, there
exist $k\in\NN$, $\eps>0$ independent of $a$, $v$ such that the following holds. Denote
$$\hat \zeta=(\hat F,\hat G)\equiv \wl{\cT^{2k}\zeta_{a,v,c}},\text{ and }\bar w=D\cT^{2k}|_{(a,v)}\bar v.$$ Then,
\begin{itemize}
\item[(c)]$\min(\inf_x\grad_{\bar w}\hat F,\inf_x\grad_{\bar w}\hat F)>(1+\eps) \min(\inf_x\grad_{\bar v}\wl F_{a,v,c},\inf_x\grad_{\bar v}\wl G_{a,v,c})$
\end{itemize}

\end{prop}

\begin{proof}
Let $\bar v=(\alpha,\nu)\in \cC_{a,v}$. Consider a smooth deformation 
\begin{equation}
\label{deform}
\zeta_t^{\bar v}=(F_{a+\alpha t,v+\nu t,c},G_{a+\alpha t,v+\nu t,c})\equiv (F_t,G_t).
\end{equation}
Set 
$$\wl F_t\equiv F_t|_{[-1,F_t^{-1}(0)]},\; \wl G_t\equiv G_t\circ F_t|_{[F_t^{-1}(0),0]}.$$

For small $t$ the heights $r_0,r_1$ of $\zeta_t^{\bar v}$ and $\cR_c\zeta_t^{\bar v}$ are constant.
Set 
$$\lambda(t)=F_t^{r_0}(-1)<0.$$
From (\ref{def-r2}), both maps in the pair $\cT_c\zeta_t^{\bar v}$ have the form
\begin{equation}
\label{derivative1}
M_t(z)=-\frac{1}{\lambda(t)}H_t(-\lambda(t)z),
\end{equation}
where $H_t$ is a composition of maps $\wl F_t$, $\wl G_t$, and $z\leq 0$.
Set $H(t,z)\equiv H_t(z)$ and $\hat H(t,z)\equiv H_t(-\lambda(t)z)$. 
We have 
$$\frac{\partial}{\partial t}\hat H(t,z)=\frac{\partial H}{\partial t}(t,-\lambda(t) z)-\frac{\partial H}{\partial z}(t,-\lambda(t) z)\cdot \lambda'(t)z.$$
Note, that $\frac{\partial H}{\partial z}>0$ since $H$ is an orientation-preserving homeomorphism, and $\lambda'(t)>0$ by our assumptions.
Since $z\leq 0$, the second term in the above sum is positive.
Furthermore, by (\ref{eqn-comp}), the first term is bounded below by 
$$\min(\inf_{x\in[-1,F_{a,v,c}^{-1}(0)]}\nabla_{\bar v}\wl F_{a,v,c},\inf_{x\in[F_{a,v,c}^{-1}(0),0]}\nabla_{\bar v}\wl G_{a,v,c})>0.$$

Finally, 
\begin{equation}
\label{derivative2}
\begin{split}
\frac{\partial}{\partial t}
M_t(z) &=
\frac{\partial}{\partial t}-\frac{1}{\lambda(t)}H_t(-\lambda(t)z)=\\
&= -\frac{1}{\lambda(t)}\frac{\partial}{\partial t}H_t(-\lambda(t)z)-\frac{\lambda'(t)}{(\lambda(t))^2}H_t(-\lambda(t)z).
\end{split}
\end{equation}
Since  $H_t<0$, the second term in (\ref{derivative2}) is positive. Combining this with the previous estimate, we have
\begin{equation}
\label{derivative3}
\frac{\partial}{\partial t}
M_t(z) \geq -\frac{1}{\lambda(t)}\min(\inf_{x\in[-1,F_{a,v,c}^{-1}(0)]}\nabla_{\bar v}\wl F_{a,v,c},\inf_{x\in[F_{a,v,c}^{-1}(0),0]}\nabla_{\bar v}\wl G_{a,v,c})>0.
\end{equation}
Statement (c) follows immediately from (\ref{derivative3}) and \propref{rab}.

Finally, let $\bar v\in\bar \cC_{a,v}$ and $\bar v\neq 0$. 
Note that $s(x)=\nabla_{\bar v}F_{a,v,c}(x)$ is a rational function of $x$ of degree $2$.
Since $\bar v\in\bar \cC_{a,v}$, the function  
$$s(x)\equiv \grad_{\bar v}F_{a,v,c}(x)\geq 0\text{ for }x\in[-1,0].$$
 As $s(x)\not\equiv 0$,
 on the interval $[0,1]$ it may have no more than $2$ zeros.
The pigeonhole principle implies that there exists 
an interval $I^2_l\subset [0,1]$ of the
$2$-nd dynamical partition of $\zeta_{a,v,c}$, in which
$\inf_{x\in I^2_l} s(x)>0$. 
By (\ref{eqn-comp}) this implies 
that 
$$\inf_z\frac{d}{dt}\wl{\cT(F_{t}, G_{t})}>0,$$
and claim (b) follows.
\end{proof}

We further show:
\begin{prop}
\label{uniform expansion}
There exists $k\in\NN$ and $\Omega>1$ such that the following holds. Let $(a,v)\in\cO^{4k}_c$ and let 
$\bar v\in\cC_{a,v}$. Denote 
$$\hat \zeta=(\hat F,\hat G)\equiv \cT^{2k}\zeta_{a,v,c},\text{ and }\bar w=D\cT^{2k}|_{(a,v)}\bar v.$$ 
Then
$$\sup_x\grad_{\bar w}\hat F>\Omega \max(\sup_x\grad_{\bar v}F_{a,v,c},\sup_x\grad_{\bar v}F_{a,v,c}\circ G_{a,v,c}).$$
\end{prop}
\begin{proof}
For simplicity sake, let us assume that 
$$\sup_x\grad_{\bar v}F_{a,v,c}\geq \sup_x\grad_{\bar v}F_{a,v,c}\circ G_{a,v,c},$$
the other possibility is handled similarly. Let 
$$x_0\equiv \operatorname{argmax}\grad_{\bar v}F_{a,v,c}.$$
Fix $m\in\NN$ and let 
$$\zeta_{2m}=(F_{2m},G_{2m})\equiv \cT^{2m}\zeta_{a,v,c}.$$
Let $$I_{2m}=[x_{2m},0]\subset [-1,0]$$ be the renormalization interval with period $q_{4m}$, so that
$$F_{2m}\equiv \alpha_{2m}\circ \zeta^{q_{4m}}_{a,v,c}|_{I_m}\circ \alpha_{2m}^{-1},\text{ where }\alpha_{2m}(x)=x/|x_{2m}|.$$
Set $I^j\equiv \zeta^i_{a,v,c}(I_{2m})$, and let $I^{j_0}\ni x_0.$
We will decompose the iterate $\zeta^{4q_m}_{a,v,c}|_{I_m}$ into
$$\zeta_{a,v,c}^{j_0}|_{I_{2m}}\circ F_{a,v,c}|_{I^{j_0}}\circ \zeta_{a,v,c}^{q_{4m}-j_0-1}|_{I^{j_0+1}}.$$
The same considerations as in the previous lemma imply that
$$\inf_x\grad_{\bar v}\zeta_{a,v,c}^{j_0}|_{I_{2m}}>0,\; \inf_x\grad_{\bar v}\zeta_{a,v,c}^{q_{4m}-j_0-1}|_{I^{j_0+1}}>0\text{, and }\grad_{\bar v}|x_{2m}|<0.$$
Let $y_0\in I_{2m}$ be such that $\zeta_{a,v,c}^{j_0}(y_0)=x_0$. A simple induction based on (\ref{eqn-comp}) implies that
$$D\zeta_{a,v,c}^{q_{4m}}(y_0)\sup_x\grad_{\bar w}\hat F>\frac{1}{|x_{2m}|}\sup_x\grad_{\bar v}F_{a,v,c}.$$
By real {\it a priori} bounds the derivative $D\zeta_{a,v,c}^{q_{4m}}(y_0)$ 
is universally bounded from below. 
Hence, there exists a universal constant $C>0$ such that 
$$\sup_x\grad_{\bar w}\hat F>\frac{C}{|x_{2m}|}\sup_x\grad_{\bar v}F_{a,v,c}.$$
By real {\it a priori} bounds, there exist $C_0>0$ and  $\lambda>1$ such that
$$\frac{1}{|x_{2m}|}>C_0\lambda^m,$$
and the claim follows.

\end{proof}

For a fixed $c$ let us denote 
$$\dist^0((a,v),(a',v'))\equiv\sup_{x\in[-1,0]}|F_{a,v,c}(x)-F_{a',v',c}(x)|.$$
The following is elementary:
\begin{prop}
The distance defined by $\dist^0$ is a smooth Riemannian metric on $\cD_c$.
\end{prop} 

\noindent
We will denote by $\cD_c^0$ the domain $\cD_c$ equipped by this smooth metric. For a vector $\bar v\in T_{(a,v)}\cD_c$ we denote
$||\bar v||^0$ the corresponding norm. 
We note the following corollary of \propref{uniform expansion}:
\begin{cor}
\label{uniform expansion 2}
There exist $k\in\NN$ and $\Omega>1$ such that the following holds. Let $(a,v)\in\cO^{4k}_c$ and let 
$\bar v\in\cC_{a,v}$. Then
$$||D\cT^{2k}(\bar v)||^0>\Omega ||\bar v||^0.$$

\end{cor}
Since any two smooth Riemannian metrics on $\cD_c$ induce equivalent norms on compact subsets, we have:
\begin{cor}
\label{uniform expansion 3}
There exist $k\in\NN$, $C>0$ and $\Omega>1$ such that the following holds. Let $(a,v)\in\cO^{4km}_c$ and let 
$\bar v\in\cC_{a,v}$. Then
$$||D\cT^{2km}(\bar v)||>C\Omega^m ||\bar v||.$$
\end{cor}
}



\section{Constructing the renormalization horseshoe}

Let us begin by making the following definition. For a finite sequence of natural numbers $r_{-2k},\ldots,r_{-1},r_0,r_1,\ldots,r_{2m}$ we set $$\cS_{c,(r_{-2k},\ldots,r_{-1},r_0,r_1,\ldots,r_{2m})}\subset \Delta_c$$ to be the set of parameters 
$(a,v)\in \cO^{2m}_c\cap \cT^{2k}(\cO^{2k})$ where
$$(a,v)=\cT^{2k}(a_{-2k},v_{-2k})\text{ such that }\rho(\zeta_{a_{-2k},v_{-2k},c})=[r_{-2k},\ldots,r_0,\ldots,r_{2m},\ldots].$$




For every inifnite sequence of natural numbers $(r_0,r_1,\ldots)$ let us denote 
$$L_{c,(r_0,r_1,\ldots)}=\cap_{k\to\infty}\cS_{c,(r_0,r_1,\ldots,r_{2k})}\subset \Delta_c,\text{ so that }$$
$$L_{c,(r_0,r_1,\ldots)}=\{(a,v)\in\Delta_c\text{ such that }\rho(\zeta_{a,v,c})=[r_0,r_1,r_2,\ldots]\in\RR\setminus\QQ\}.$$
We naturally identify the tangent space $T_{a,v}\Delta_c$ with $\RR^2$ by 
$$\iota:x\frac{\partial }{\partial a}+y\frac{\partial }{\partial v}\mapsto (x,y).$$
\begin{prop}
\label{curve1} Let $c\in(0.5,2)\setminus\{ 1\}$.
Then every set $L_{c,(r_0,r_1,\ldots)}$ is a continuous curve. 
Furthermore, for every point $(a,v)\in L_{c,(r_0,r_1,\ldots)}$ and every vector $\bar v\in \cC_{a,v}$,
 the curve is locally disjoint from the line segment $\iota(\bar v)\setminus \{(a,v)\}$.
\end{prop}

\begin{proof}
Let us start with the second assertion. Assume that for some $(a,v)\in L_{c,(r_0,r_1,\ldots)}$ there exists a point
$(a',v')\in L_{c,(r_0,r_1,\ldots)}\cap  \iota(\cC_{a,v})$. Denote $L$ the line segment connecting $(a,v)$ and $(a',v')$.
Consider a linear map $\gamma:[0,1]\to L$ with $\gamma(0)=(a,v)$ and $\gamma(1)=(a',v')$. 
Assuming that $(a',v')$ is sufficiently close to $(a,v)$, we have 
$$\frac{d}{dt}\bar \gamma(t)\in\cC_{a(t),v(t)}\text{ for every }t\in[0,1]$$
by considerations of continuity. Hence, by \propref{rho-monotone} we have 
$$\rho(\zeta_{a,v,c})<\rho(\zeta_{a',v',c}),$$
which contradicts our assumption.

Let us show that the set $L_{c,(r_0,r_1,\ldots)}$ is a continuous curve. 
We first present the argument for $c>1$. 
Note that for every $(a_0,v_0)\in\Delta_c$ the vertical line 
$$\{v=v_0\}\subset \iota(\cC_{a_0,v_0}).$$ Thus, the intersection
$$\alpha(v_0)\equiv L_{c,(r_0,r_1,\ldots)}\cap \{v=v_0\}$$
contains at most one point. 

Furthermore, 
consider the domain $\cD_c\setminus \cO^1_c$. On the upper boundary, $\{a=c\}$ we have 
$F_{a,v,c}(-1)=G_{a,v,c}(a)=0$. Hence, the rotation number 
$\rho(\zeta_{a,v,c})=[1,\infty]=1.$ On the other hand, on the lower boundary curve every $\zeta_{a,v,c}$ has a fixed point in $[-1,0]$,
which is either the boundary point $-1$, or a point $w\in(-1,0)$ with $F_{a,v,c}'(w)=1$. Thus the rotation number 
$\rho(\zeta_{a,v,c})=[\infty]=0$. Hence, by Intermediate Value Theorem, for $v>-1$ we have 
$$\alpha(v)\neq \emptyset.$$
Finally, by continuity in the dependence of the rotation number on parameters, the parametrization
$$v\mapsto \alpha(v)\in L_{c,(r_0,r_1,\ldots)}$$
is continuous.

The proof for $c\in(0.5,1)$ is completely analogous , with the vertical lines $\{v=v_0\}$ replaced 
by $\{a=a_0\exp(v/c)\}$. We leave the details to the reader.
\end{proof}

\ignore{

\begin{defn}
For a finite continued fraction $[r_0,\ldots,r_{4n}]=p/q$ let us denote 
$L_{c,[r_0,\ldots,r_{4n}]}$ the collection of parameters $(a,v)\in\cS_{c,(r_0,\ldots,r_{4n})}$ for which the critical point $0$ is
periodic, in other words, $f_{a,v,c}^q(s_a(0))=s_a(0)$.
\end{defn}

We prove:
\begin{prop}
\label{critical curves} Let $c>0.5$. Then each set $L_{c,[r_0,\ldots,r_{4n}]}$ is an analytic curve. Moreover, for every 
$(a,v)\in L_{c,[r_0,\ldots,r_{4n}]}$ the tangent line $\ell\subset T_{a,v}\cD_c$ to this curve lies outside of $\cC_{a,v}$.
\end{prop}
\begin{proof}
Let us present the argument for $c>1$, the proof for $c\in(0,5,1)$ being completely analogous. By \propref{rho-monotone}
for every $(a_0,v_0)\in\cD_c$ the set 
$$M_{v_0}\equiv \{(a,v_0)\in\cD_c\;|\; \rho(\zeta_{a,v_0,c})=[r_0,\ldots,r_{4n}]\}$$
is a closed interval $$[a_l(v_0),a_u(v_0)]\times\{ v_0\}.$$ Let $$\hat\zeta=(\hat F,\hat G)\equiv \cT^{2n}\zeta_{a,v,c}.$$
Since $\frac{\partial}{\partial a}\in \cC_{a,v}$, the critical value $\hat F(-1)=\hat G(0)$ strictly increases
as a function of $a$ inside the interval $M_{v_0}$, with $\hat F(-1)=0$ at the upper boundary point $(a_u(v_0),v_0)$  of $M_{v_0}$. Thus,
the intersection 
$$L_{c,[r_0,\ldots,r_{4n}]}\cap \{v=v_0\}=a_u(v_0).$$
By continuity of the rotation number, the function $a_u(v_0)$ is continuous. Finally, the analyticity of the condition
$f_{a,v,c}^q(s_a(0))=s_a(0)$ together with the Implicit Function Theorem imply that $a_u(v_0)$ is analytic.

\end{proof}

}
We now use the Duality Theorem. 
For every backward-inifnite sequence of natural numbers $(\ldots,r_{-k},\ldots,r_{-2},r_{-1})$ let us denote 
$$L_{c,(\ldots,r_{-2},r_{-1})}=\cap_{k\to\infty}\cS_{c,(r_{-2k-1},\ldots,r_{-1})}\subset \Delta_c.$$
As a corollary of the Duality Theorem, we have:
\begin{prop}
\label{curve2}
Every $$L_{1/c,(\ldots,r_{-2},r_{-1})}=\cI_c(L_{c,(r_{-1},r_{-2},\ldots)}).$$ 
In particular, every $L_{1/c,(\ldots,r_{-2},r_{-1})}$ is a continuous curve. 
\end{prop}

We are now ready to show:
\begin{thm}
\label{attractor4}
Let $c\in(0.5,2)\setminus\{ 1\}$, and let $(r_i)_{i\in\ZZ}$ be a periodic sequence of positive integers with period
$2p,\; p\in\NN$. Then there exists a unique $p$-periodic  point  $(a_*,v_*)$ of $\cT_c$ with the property
$$\rho(\zeta_{a_*,v_*,c})=[r_0,r_1,\ldots,r_{2p-1},\ldots]\equiv \rho.$$
Furthermore, the orbit $(a_*,v_*)$ is hyperbolic with one stable and one unstable directions.

Finally, the curves $$L_{c,(r_0,r_1,\ldots)}=W^s(a_*,v_*),\text{ and }L_{c,(...,r_{-k},\ldots,r_{-1})}=W^u(a_*,v_*).$$

\end{thm}
\begin{proof}
Elementary considerations of  compactness and convergence imply that there exists a non-empty compact set 
$$\Omega\subset L_{c,(r_0,r_1,\ldots)}$$
such that $(\cT_c)^p(\Omega)=\Omega$, and moreover, for every $\zeta_{a,v,c}$ with $\rho(\zeta_{a,v,c})=\rho$
we have
$$(\cT^c)^{pk}(\zeta_{a,v,c})\underset{k\to\infty}{\longrightarrow}\Omega.$$
Since a continuous mapping of a closed interval always has a fixed point, 
there exists at least one $p$-periodic point in $\Omega$, let us denote it $(a_*,v_*)$.

By \propref{expcone1}, the matrix $D\cT_c^p|_{(a_*,v_*)}$ has one expanding eigenvalue. By the Duality Theorem, it has 
a contracting eigenvalue. Hence, $(a_*,v_*)$ is a hyperbolic periodic orbit. 

Let us prove that $L_{c,(r_0,r_1,\ldots)}=W^s(a_*,v_*)$ for $c\in(1,2)$. 
By continuity of the dependence of the rotation number on parameter,
$$W^s_{\text{loc}}(a_*,v_*)\subset L_{c,(r_0,r_1,\ldots)}.$$
Furthermore, both sets are continuous curves and hence
locally coincide. 

By construction of the renormalization operator, the global stable manifold
$W^s(a_*,v_*)$ is a smooth sub-arc of $L_{c,(r_0,r_1,\ldots)}.$
Suppose that $W^s(a_*,v_*)$  is not the whole curve $L_{c,(r_0,r_1,\ldots)}$, and thus $W^s(a_*,v_*)$ has an endpoint $(x,y)\notin W^s(a_*,v_*)$.
By invariance of $W^s(a_*,v_*)$ under $\cT^p$, the point $(x,y)$ is also a $p$-periodic point of $\cT$. The same argument as above implies that
it is hyperbolic and that $W^s_{\text{loc}}(x,y)$ is a smooth sub-arc in $L_{c,(r_0,r_1,\ldots)}$. Hence 
$$W^s_{\text{loc}}(x,y)\cap W^s(a_*,v_*)\neq \emptyset,$$
and we have arrived at a contradiction.

The argument for $c\in(0.5,1)$ is again completely analogous, with the lines $\{v=v_0\}$ replaced 
by curves $\{a=a_0\exp(v/c)\}$. 

As a consequence, 
$$\Omega=\{(a_*,v_*)\}.$$
Finally,
the statement $L_{c,(...,r_{-k},\ldots,r_{-1})}=W^u(a_*,v_*)$ follows by the Duality Theorem.

\end{proof}

Let us formulate a few corollaries:
\begin{prop}\label{attractor5}
For every periodic sequence $(r_i)_{i\in\ZZ}$, the curves
 $L_{c,(r_0,r_1,\ldots)}$ and $L_{c,(...,r_{-k},\ldots,r_{-1})} $ are $C^\omega$-smooth.
\end{prop}

The following statement follows from the results of \cite{KT2}:
\begin{prop}
\label{transverse}
For every periodic sequence 
$(r_i)_{i\in\ZZ}$, the curves
 $L_{c,(r_0,r_1,\ldots)}$ and $L_{c,(...,r_{-k},\ldots,r_{-1})} $ intersect uniformly transversely.
\end{prop}
Note (see Remark \ref{rem:bdedtype}) that for orbits of bounded type the uniformity of transversality of intersection follows
by considerations of compactness, without appealing to \cite{KT2}.

\begin{prop}
\label{attractor1}
For every $c\in(0.5,2)\setminus \{1\}$ and every bi-infinite sequence of natural numbers
$(r_i)_{i\in\ZZ}$ there exists a unique point $(a,v)=(a,v)_{(r_i)_{i\in\ZZ}}\in\Delta_c$ such that:
\begin{itemize}
\item $\rho(\zeta_{a,v,c})=[r_0,r_1,\ldots]$;
\item for every $m\in\NN$ the rotation number 
$$\rho(\cT_c^{-m}(\zeta_{a,v,c}))=[r_{-2m},\ldots,r_0,r_1,\ldots].$$
\end{itemize}
\end{prop}
\begin{proof}
Let us show that the intersection of the curves
$$\Omega\equiv L_{c,(\ldots,r_{-2},r_{-1})}\cap L_{c,(r_0,r_{-1},\ldots})$$ is non-empty.
For every even $n\in\NN$ let 
$$\rho_n\equiv [r_{-n},\ldots,r_{n-1},r_{-n},\ldots,r_{n-1},\ldots],$$
and let $(a_n,v_n)$ be the unique $\cT_c$-periodic point with period $n$ and 
$$\rho(\zeta_{a_n,v_n,c})=\rho_n.$$

Then, by continuity of the rotation number, every limit point of the sequence 
$\{(\cT_c)^{n/2}(a_n,v_n)\}$ lies in $\Omega$. Thus, 
$\Omega\neq\emptyset.$

Furthermore, let us show that for each bi-infinite sequence $(r_i)_{i\in\ZZ}$
the intersection $L_{c,(\ldots,r_{-2},r_{-1})}\cap L_{c,(r_0,r_{-1},\ldots})$ consists of a single point. We first note that in the case when 
$(r_i)_{i\in\ZZ}$ is a periodic sequence, the curves $L_{c,(\ldots,r_{-2},r_{-1})}$ and $L_{c,(r_0,r_{-1},\ldots})$ are the unstable and stable
manifolds of the unique periodic orbit with rotation number $[r_0,r_1,\ldots]$. If these manifolds have a homoclinic intersection point, 
then neither one of them could be a smooth curve, which would contradict \thmref{attractor4}.

The general case now follows by \propref{transverse} and considerations of continuity.


\end{proof}

Let us denote $\cA_c$ the collection of all pairs $(a,v)_{(r_i)_{i\in\ZZ}}\in\Delta_c$. By construction,
we have:

\begin{prop}
\label{attractor2}
The map $\iota:\Sigma_\NN\to \cA_c$ given by
$\iota:(r_i)_{i\in\ZZ}\mapsto (a,v)_{(r_i)_{i\in\ZZ}}\in\Delta_c$ is a homeomorphism.
\end{prop}

By \propref{attractor1},
\begin{equation}
\label{attractor3}
\rho(\zeta_{\iota(r_i),c})=[r_0,r_1,\ldots]\text{ and }\iota^{-1}\circ \cT_c\circ \iota=\sigma^2.
\end{equation}

Denote $\cP_c$ the dense subset of $\cA_c$ consisting of periodic orbits. For every periodic sequence $(r_i)_{i\in\ZZ}$,
consider the tangent vector field $T^s$ of the curves $L_{c,(r_0,r_1,\ldots)}$, and the tangent vector field 
$T^u$ of  $L_{c,(...,r_{-k},\ldots,r_0)} $. 
By a simple diagonal convergence process, 
$T^s$ and $T^u$ 
can be extended to all of $\cA_c$ as a $D\cT_c$-invariant splitting of the tangent bundle. Furthermore, 
\propref{expcone1} implies that $D\cT$ uniformly expands $T^u$, and by Duality Theorem, $D\cT$ uniformly contracts $T^s$. Hence:
\begin{prop}
For $c\in(0.5,2)\setminus\{ 1\}$ the set $\cA_c$ is uniformly hyperbolic, with one stable and one unstable direction. The curves
$L_{c,(r_0,r_1,\ldots)}$ and $L_{c,(...,r_{-k},\ldots,r_0)} $ are the stable and the unstable foliation of $\cA_c$ respectively.
\end{prop}
Thus,
\begin{prop}
\label{prop:analytic}
The curves $L_{c,(r_0,r_1,\ldots)}$ and $L_{c,(...,r_{-k},\ldots,r_0)} $ are $C^\omega$-smooth.
\end{prop}